\documentclass[a4paper,10pt]{article}
\usepackage{latexsym,amssymb,amsfonts,amsmath,amsthm,amstext,color,graphicx,times}
\usepackage[mathscr]{euscript}
\usepackage{indentfirst}
\usepackage{cite}
\usepackage{mathrsfs}
\usepackage{subcaption}

\newcommand{\R}{\mathbb{R}}
\newcommand{\Q}{\mathbb{Q}}
\newcommand{\N}{\mathbb{N}}

\usepackage{wrapfig}
\usepackage[all]{xy}
\usepackage{tikz}
\usepackage{hyperref}
\usepackage{animate}
\usetikzlibrary{patterns}

\theoremstyle{plain}
\newtheorem{theorem}{Theorem}[section]
\newtheorem{lemma}[theorem]{Lemma}
\newtheorem{proposition}[theorem]{Proposition}

\theoremstyle{definition}

\newtheorem{example}[theorem]{Example}

\theoremstyle{remark}
\newtheorem{remark}[theorem]{Remark}

\DeclareMathOperator{\gra}{gra}

\newcommand{\tos}{\rightrightarrows} 
\DeclareMathOperator*{\argmin}{arg\,min}

\DeclareMathOperator*{\INT}{int}

\DeclareMathOperator*{\conv}{conv}

\DeclareMathOperator*{\nep}{NEP}

\DeclareMathOperator*{\gnep}{GNEP}

\DeclareMathOperator{\linear}{span}
\newcommand{\inner}[2]{\langle #1,#2 \rangle}
\title{A note on  coupled constraint Nash games}

\author{Orestes Bueno\thanks{Universidad del Pac\'ifico.  Lima, Per\'u. Email: \texttt{\{o.buenotangoa, cotrina\_je\}@up.edu.pe}} 
\and Carlos Calderon\thanks{Instituto de Matem\'atica y Ciencias Afines. Email: \texttt{carlos.calderon@imca.edu.pe}}
\and John Cotrina\footnotemark[1]}

\begin{document}

\maketitle

\begin{abstract}
In this note we are interested in a relevant generalized Nash equilibrium problem, which was proposed by Rosen in 1965.  
An existence result is established in the general setting of quasiconvexity, which is independent from the one given by Aussel and Dutta in 2008.

\bigskip

\noindent{\bf Keywords: }
	Generalized Nash Equilibrium; Generalized Convexity; Variational Inequalities

\bigskip

\noindent{{\bf MSC (2010)}: 49J53, 91A10, 91B50 }
\end{abstract}
\section{Introduction}
Let $N$ be a finite and non-empty set, representing a set of players.
Let us assume that each player $\nu \in N$ chooses a strategy $x^\nu $ in a strategy set $K_{\nu}$ of  $\R^{n_\nu}$.Consider the sets
\[
K=\prod_{\nu\in N} K_\nu,\quad \R^n=\prod_{\nu\in N}  \R^{n_\nu},\quad  K_{-\nu}=\prod_{\mu\in N\setminus\{\nu\}} K_\mu,\quad \R^{-n_\nu}=\prod_{\mu\in N\setminus\{\nu\}} \R^{n_\mu}
\]
Also, we can write $x = (x^\nu, x^{-\nu}) \in K$ in order to emphasize the strategy of player $\nu$, $x^\nu\in K_\nu$, and the strategy of the rival players $x^{-\nu}\in K_{-\nu}$.

For each player $\nu$, given the rivals' strategy $x^{-\nu}$, the player $\nu$ chooses a strategy $x^\nu$ such that it solves the following optimization problem  
\begin{equation}\label{NEP}
	\textstyle{\nep_{\nu}}(x^{-\nu})\::\:\displaystyle{\min_{K_\nu} \theta_\nu(\cdot,x^{-\nu})} ,
\end{equation}
where $\theta_\nu: \R^n\to\R$ and  $\theta_\nu(x^\nu,x^{-\nu})$ denotes the \emph{loss} that player $\nu$ suffers when his own strategy is $x^\nu$ and the rivals' strategies are $x^{-\nu}$.  Thus, a \emph{Nash equilibrium}~\cite{Nash}, associated to the loss functions $(\theta_\nu)_\nu$ and the constraint sets $(K_\nu)_\nu$, is a vector $\hat{x}\in K$ such that $\hat{x}^\nu$  solves $\nep_\nu(\hat x^{-\nu})$, for any $\nu$.

It is well known~\cite[Theorem 1]{Dansupta} in the literature that a Nash equilibrium exists when, for each $\nu\in N$,
\begin{itemize}
	\item the set $K_\nu$ is compact convex and non-empty,
	\item the function $\theta_\nu$ is continuous and is quasiconvex on its player's variable.
\end{itemize}
Since then, many authors extended this result in different ways, see for instance~\cite{Dansupta,Tian95,Reny,Rabia,Banks-Duggan}, where continuity assumption is relaxed, or \cite{Kazuo,PQK-NHQ-19}, where the quasiconvexity is weakened.

On the other hand, in a Generalized Nash Equilibrium Problem (GNEP), each player $\nu$ posseses a \emph{constraint map} $\mathfrak X_\nu:K_{-\nu}\tos K_{\nu}$, so their strategy must belong to the set $\mathfrak X_\nu(x^{-\nu})\subset K_\nu$, where the sets $\mathfrak X_\nu(x^{-\nu})$ clearly depend on the rival players' strategies. The aim of player $\nu$, given the rival players' strategies $x^{-\nu}$, is to choose a strategy $x^\nu$ that solves the minimization problem
\begin{equation}\label{GNEP}
	\textstyle{\gnep_{\nu}}(x^{-\nu})\::\:\displaystyle{\min_{\mathfrak X_\nu(x^{-\nu})} \theta_\nu(\cdot,x^{-\nu})}.
\end{equation}
A \emph{Generalized Nash Equilibrium} is a vector $\hat{x}\in K$ such that $\hat{x}^\nu$ solves $\gnep_\nu(\hat x^{-\nu})$, for any $\nu$.
Arrow and Debreu~\cite[Lemma 2.5]{Arrow-Debreu} proved the existence of Generalized Nash Equilibria when for all $\nu \in N$,
\begin{itemize}
	\item the set $K_\nu$ is compact convex and non-empty,
	\item the function $\theta_\nu$ is continuous and quasiconvex on its player's variable,
	\item the map $\mathfrak X_\nu$ is lower semicontinuous with closed graph and convex and non-empty values.
\end{itemize}
Clearly this existence result generalizes the first one, by considering $\mathfrak X_\nu$ as the constant map $\mathfrak X_{\nu}(x^{-\nu})=K_\nu$.  Since then, many authors dealt with this kind of problems, see for instance~\cite{HARKER199181,Tian92,Tian95,Cubiotti,ASV-2016,AGM-16,JC-JZ-PS,Au-Su,JC-JZ,JC-AS,Cotrina2020ExistenceOQ}. 

An important instance of GNEPs was presented by Rosen in~\cite{Rosen}, which is known as the \emph{jointly convex GNEP}. 
A jointly convex GNEP is a GNEP in which the graph of each map constraint $\mathfrak X_\nu$ is a certain convex, compact and non-empty subset of $\R^n$, denoted by $\mathfrak X$. In other words
\[
\mathfrak X_\nu(x^{-\nu})=\{x^\nu\in\R^{n_\nu}\::\:(x^\nu,x^{-\nu})\in \mathfrak X\},
\]
When this is the case, $K_\nu$ is the projection of $\mathfrak X$ onto $\R^{n_\nu}$. This kind of GNEP have recently begun to gain more and more attention as it models problems from economics such as electricity markets, environmental games and bilateral exchange of bads, see~\cite{CKK-2004,K-2005,CK-2015,K-2016,vardar2019}.

On the other hand, jointly convex GNEPs can
be reformulated as variational inequalities~\cite{FACCHINEI2007159} when assuming pseudoconvexity and differentiability. An extension of this reformulation was presented by Aussel and Dutta in \cite{Aussel-Dutta}, using normal cones, under quasiconvexity and continuity assumptions. However, their main result concerning the existence of equilibria was established under semi strict quasiconvexity.

In this note, our main result consists to assume only quasiconvexity in order to show the existence of equilibria for the jointly convex GNEP. In Section~\ref{definitions} we recall some definitions of generalized convexity, continuity for functions and continuity for set-valued maps. Additionally some preliminary results are established concerning closedness and lower semicontinuity for special maps. Finally, in Section~\ref{main-section} we reformulate jointly convex GNEPs as variational inequality problems and show our main result.

\section{Definitions, Notations and Preliminary Results}\label{definitions}
Let $U,V$ be non-empty sets. A \emph{multivalued} or \emph{set-valued} map $T:U\tos V$ is an application $T:U\to \mathcal{P}(V)$, that is, for $u\in U$, $T(u)\subset V$. 
The  \emph{graph} of $T$ is defined as
\begin{gather*}
	\gra(T)=\big\{(u,v)\in U\times V\::\: v\in T(u)\big\}.
\end{gather*}

Let $X,Y$ be two real Banach spaces. Recall that a set-valued map $T:X\tos Y$ is
\begin{itemize}
	\item  \emph{lower semicontinuous} (lsc) at $x_0\in X$ when, for any open set $V$ such $T(x_0)\cap V\neq\emptyset$,  there exists a neighborhood $U$ of $x_0$ such that $T(x)\cap V\neq\emptyset$, for every $x\in U$;
	\item \emph{upper semicontinuous} (usc) at $x_0\in X$ when, for any neighborhood $V$ of $T(x_0)$, there exists a neighborhood $U$ of $x_0$ such that $T(x)\subset V$, for all $x\in U$;
	\item  \emph{lower} (respectively, \emph{upper}) \emph{semicontinuous}  when it is lower (resp. upper) semicontinous at every $x_0\in X$;
	\item \emph{closed} when its graph is a closed subset of $X\times Y$.
\end{itemize}
Moreover, $T:X\tos Y$ is lower semicontinuous at $x_0$ if, and only if, $T$ is \emph{sequentially lower semicontinuous} at $x_0$, that is, for any sequence $(x_n)_{n\in\N}$ of $X$ converging to $x_0$ and any $y_0\in T(x_0)$, there exists a sequence $(y_n)_{n\in\N}$ of $Y$ converging to $y_0$ such that $y_n\in T(x_n)$, for any $n$. The interested reader can find a proof of this fact in~\cite[Proposition 2.5.6]{VMPOS}.

Let $X$ be a Banach space, $X^*$ its topological dual and $\langle\cdot,\cdot\rangle$ the duality pairing. Given $A$ a subset of $X$, the \emph{closure}, the \emph{interior} and the \emph{convex hull} of $A$ are denoted by $\bar{A}$, $\INT(A)$ and $\conv(A)$, respectively. 
In addition, $-A$ denotes the set $-A=\{x\in X\::\: -x\in A\}$.
The \emph{polar cone} of $A$ is the set 
\[
A^-:=\{x^*\in X^*\::\: \inner{x^*}{x}\leq 0,\,\forall\,x\in A\},
\]
which is a convex cone on $X^*$.
The \emph{normal cone} $N_A$ of $A$ at $x\in X$ is defined as
\[
N_{A}(x):=\{x^*\in X^*\::\: \inner{x^*}{y-x}\leq 0,\,\forall\, y\in A\}.
\]
We will consider $N_A(x)=X^*$, whenever $A=\emptyset$. It is common in this definition to impose $A$ convex and $x\in A$, however, we won't consider such conditions in this work. See~\cite{Bueno-Cotrina} for another instance where such conditions are not considered.
Note that, for all $x\in X$,  
\[
N_A(x)=(A-x)^-=N_{\bar{A}}(x)=N_{\conv(A)}(x).
\]

The following lemma will be used in the proof of Proposition~\ref{Prop1}. 
\begin{lemma}\label{L0}
	Let $A$ be a subset of a real Banach space $X$. If $\INT(A)\neq\emptyset$, then $A^-$ is a pointed cone, i.e. $A^-\cap -A^-=\{0\}$. Moreover, if $x \in \INT(A)$ and $x^*\in A^-$, $x^*\neq 0$, then $\inner{x^*}{x}< 0$.
\end{lemma}
\begin{proof}
	Note that $A^-\cap -A^-=A^{\perp}=\{x^*\in X^*\::\: \inner{x}{x^*}=0,\,\forall\,x\in A\}$. Since $\INT(A)\neq \emptyset$, $\linear(A)=X$, so
	\[
	A^-\cap -A^- = A^{\perp} = \linear(A)^{\perp} = X^{\perp}=\{0\}.
	\]
	To conclude the proof, let $x\in \INT(A)$, $x^*\in A^-$, $x^*\neq 0$ and assume that $\inner{x}{x^*}\geq 0$, hence $\inner{x}{x^*}=0$. Since $x\in\INT(A)$, for some $\varepsilon>0$ and all $v\in B(0,1)$, $x+\varepsilon v\in A$, which implies $\inner{x+\varepsilon v}{x^*}\leq 0$. Therefore $\inner{v}{x^*}\leq 0$, for all $v\in B(0,1)$, which in turn implies that $\inner{v}{x^*}=0$, for all $v\in B(0,1)$. This is a contradiction, since $x^*\neq 0$. 
\end{proof}


Let us now recall some classical definitions of generalized convexity. Given a Banach space $X$, a real-valued function $h:X\to\R$ is
\begin{itemize}
	\item \emph{quasiconvex} if, for any $x,y\in X$ and $t\in[0,1]$,
	\[
	h(tx+(1-t)y)\leq \max\{h(x),h(y)\};
	\]
	\item \emph{semistrictly quasiconvex} if it is quasiconvex and, 
	for any $x,y\in X$ and $t\in]0,1[$, 
	\[
	h(x)\neq h(y)\quad \Rightarrow \quad h(tx+(1-t)y)< \max\{h(x),h(y)\}.
	\]
	
\end{itemize}
Our definition of semistrictly quasiconvex function was taken from~\cite[Chapter~5]{MR3236910}, and it is equivalent to the definition of pseudoconvex function given in~\cite[Definition~3.1]{SN10}. Some authors may not include the quasiconvex condition in this definition, see for instance~\cite[Chapter~4]{MR3236910}. 

It is well known that a function $h$ is quasiconvex if, and only if, the \emph{sublevel sets} $S_{h,\lambda}=\{y\in X\::\:h(y)\leq\lambda\}$ are convex, for all $\lambda\in\R$, and also if and only if the \emph{strict sublevel sets} $S^<_{h,\lambda}=\{y\in X\::\:h(y)<\lambda\}$ are convex, for all $\lambda\in\R$.

Let $X$ be a Banach space. Given $x\in X$, $\mathscr{V}_x$ will denote an open neighbourhood of $x$ in $X$. A function $h:X\to\R$ is called: 
\begin{itemize}
	\item \emph{lower semicontinuous}, if for all $x\in X$ and $\lambda\in\R$ with $\lambda <h(x)$, there exists $\mathscr{V}_x$ such that $\lambda < h(x')$, for all $x'\in \mathscr{V}_x$;
	\item \emph{upper semicontinuous}, if for all $x\in X$ and $\lambda\in\R$ with $\lambda >h(x)$, there exists $\mathscr{V}_x$ such that $\lambda > h(x')$, for all $x'\in \mathscr{V}_x$; 
	\item \emph{lower pseudocontinuous} \cite{MORGAN2007},  if for all $x,y\in X$ with $h(y)<h(x)$, there exists $\mathscr V_x$ such that $h(y)<h(x')$, for all $x'\in V_x$;
	\item \emph{upper pseudocontinuous},  if for all $x,y\in X$ with $h(y)>h(x)$, there exists $\mathscr V_x$ such that $h(y)>h(x')$, for all $x'\in V_x$.
\end{itemize}
It is not difficult to prove that if $h$ is lower semicontinuous then it is lower pseudocontinuous. On the other hand,  it is well known that $h$ is lower semicontinuous if, and only if, $S_{h,\lambda}$ is closed, for all $\lambda\in\R$. In the same way, considering for  $x\in X$  the sets $S_{h}(x):=S_{h,h(x)}$, $S^<_{h}(x):=S^<_{h,h(x)}$,  $h$ is lower pseudocontinuous if, and only if, $S_h(x)$ is closed, for all $x\in X$. Also, $h$ is upper pseudocontinuous if, and only if, $S^<_h(x)$ is open, for all $x\in X$.


Let $X$, $Y$ be two real Banach spaces and $f:X\times Y\to\R$ be a  function. We define the following set-valued maps  $L^<_f,L_f:X\times Y\tos X$ as
\begin{equation}\label{Op-L}
	L^<_f(x,y):=S^<_{f(\cdot,y)}(x)\quad\text{and}\quad L_f(x,y):=S_{f(\cdot,y)}(x).
\end{equation}
and the set-valued maps  $\mathscr{N}_f^<,\mathscr{N}_f:X\times Y\to X^*$
\begin{equation}\label{Op-N}
	\mathscr{N}_f^<(x,y):=N_{L^<_f(x,y)}(x)\quad\text{and}\quad\mathscr{N}_f(x,y):=N_{L_f(x,y)}(x).
\end{equation}

It is important to note that $L^<_f(x,y)$ is not merely the projection onto $X$ of $S^<_f(x,y)$. Consider for instance the function $f:\R\times\R\to\R$, $f(x,y)=x^2+y^2$, and take $(0,1)$. It is not difficult to see that 
$L_f^<(0,1)=\emptyset$, but the projection of $S^<_f(0,1)$ onto $\R$ is the interval $]-1,1[$. 

In general, we have the following relation between $L_f^<(x,y)$ and $S_f^<(x,y)$:
\[
L^<_f(x,y)\times\{y\}= S^<_{f(x,y)}\cap (X\times \{y\}).
\]

The following result establishes the lower semicontinuity of $L^<_f$, which is inspired by \cite{quasiconvex}.
\begin{lemma}\label{L1}
	Let $X, Y$ be two real Banach spaces and $f:X\times Y\to\R$ be a function. If $f$ is lower pseudocontinuous on its first argument and continuous on its second one, then the map $L^<_f$ is lower semicontinuous.
\end{lemma}
\begin{proof} 
	Assume, on the contrary, that $L^<_f$ is not lower semicontinuous at some $(x_0,y_0)\in X\times Y$. Then, there exists an open set $V\subset X$ with $L^<_f(x_0,y_0)\cap V\neq \emptyset$ and, for every $n,m\in\N$, there exist $x_n\in B(x_0,1/n)$, $y_m\in B(y_0,1/m)$, such that 
	\[
	L^<_f(x_n,y_m)\cap V=\emptyset. 
	\]
	This means that $f(z,y_m)\geq f(x_n,y_m)$, for all $z\in V$ and $n,m\in\N$. Since $f$ is continuous on its second argument, taking limit when $m\to\infty$, we obtain $f(z,y_0)\geq f(x_n,y_0)$, that is, $x_n\in L_f(z,y_0)$, for all $n\in\N$. We now take limit when $n\to\infty$ to conclude $x_0\in L_f(z,y_0)$, since $L_f(z,y_0)$ is closed, as $f$ is lower pseudocontinuous on its first argument. Therefore $f(x_0,y_0)\leq f(z,y_0)$, for all $z\in V$. The lemma follows by observing that, $f(z_0,y_0)<f(x_0,y_0)$, for some $z_0\in V$.
\end{proof}

\begin{lemma}\label{L2}
	Let $X, Y$ be two real Banach spaces and $T:X\times Y\tos X$ be a set-valued map.
	If $T$ is lower semicontinuous, then the set-valued map $\mathcal{N}_T:X\times Y\tos X^*$ defined as
	\[
	\mathcal{N}_T(x,y):=N_{T(x,y)}(x)
	\]
	is closed.
\end{lemma}
\begin{proof}
	Let $(x_n,y_n,x_n^*)_{n\in\N}$ be a sequence in the graph of $\mathcal{N}_T$ converging to $(x_0,y_0,x_0^*)$. We aim to show that $x_0^*\in \mathcal{N}_{T(x_0,y_0)}(y_0)$. If $T(x_0,y_0)=\emptyset$, there is nothing to prove. Now, assume that $T(x_0,y_0)\neq\emptyset$ and take any $x\in T(x_0,y_0)$. By lower semicontinuity of $T$ there exists a sequence $(z_n)_{n\in\N}$ converging to $x$ such that $z_n\in T(x_n,y_n)$, for all $n\in\N$. Therefore
	\[
	\inner{x_n^*}{z_n-x_n}\leq0.
	\]
	The lemma follows by letting $n$ tend to $\infty$, to obtain $\inner{x_0^*}{x-x_0} \leq0$. 
\end{proof}

\begin{proposition}\label{Prop1}
	Let $X, Y$ be two real Banach spaces and $f:X\times Y\to\R$ be a function. 
	\begin{enumerate}
		\item If $f$ is lower pseudocontinuous on its first argument and continuous on its second one, then the map $\mathscr{N}^<_f$
		is closed.
		\item If $f$ is upper pseudocontinuous on its first argument then $L^<_f(x,y)$ is open, for all $(x,y)\in X\times Y$. In particular, if $w^*\in \mathscr N^<_f(x,y)$, $w^*\neq 0$, then $\inner{w^*}{z-x}<0$, for all $z\in L^<_f(x,y)$.
		\item If $f$ is upper pseudocontinuous and quasiconvex, both on its first argument, then $\mathscr{N}^<_f(x,y)\neq\{0\}$, for all $(x,y)\in X\times Y$.
		
	\end{enumerate}
\end{proposition}
\begin{proof}
	\begin{enumerate}
		\item Closedness of $\mathscr{N}^<_f$ follows from  Lemmas~\ref{L1} and~\ref{L2}, and the fact that
		$\mathscr{N}^<_f(x,y)=\mathcal{N}_{L^<_f}(x,y)$, for any $(x,y)\in X\times Y$.
		\item The set $L^<_f(x,y)$ is open, since $L^<_f(x,y)=S^<_{f(\cdot,y)}(x)$ and $f(\cdot,y)$ is upper pseudocontinuous. The second claim follows from the fact that $\mathscr N^<_f(x,y)=(L^<_f(x,y)-x)^-$ and Lemma~\ref{L0}.
		\item Take $(x,y)\in X\times Y$. If $x\in \argmin_X f(\cdot,y)$ then $L^<_f(x,y)=\emptyset$ and the claim trivially follows. Now assume that $x\notin \argmin_X f(\cdot,y)$. The quasiconvexity of $f(\cdot,y)$ implies $S^<_{f(\cdot,y)}(x)=L^<_{f}(x,y)$ is non-empty and convex. Moreover $x\notin L^<_f(x,y)$ so, using the Hahn-Banach separation theorem, there exists $x^*\in X^*$, $x^*\neq 0$ such that
		\[
		\inner{x^*}{z}\leq \inner{x^*}{x},\quad \forall z\in L^<_f(x,y). 
		\]
		This implies that $x^*\in \mathscr{N}^<_f(x,y)$.\qedhere
	\end{enumerate}
\end{proof}
\begin{remark}\label{rem:finite}
	The pseudocontinuity condition on $f$ in item 3 can be dropped if $X$ is finite-dimensional.
\end{remark}

The following example shows that continuity and convexity only in the first argument of $f$ is not enough  to guarantee the closedness of $\mathscr{N}^<_f$, neither the lower semicontinuity of $L^<_f$.
\begin{example}\label{E1}
	Define $\Theta:\R\times\R\to\R$ as
	\[
	\Theta(x,y)=
	\begin{cases}
		x,&\text{when }y\neq 1,\\
		-x,&\text{when }y= 1.
	\end{cases}
	\] 
	It is not difficult to see that $\Theta$ is  continuous and convex with respect to its first argument. Moreover, for each $(x,y)\in\R^2$,
	\[
	L^<_\Theta(x,y)=
	\begin{cases}
		]-\infty,x[,&\text{when }y\neq 1,\\
		]x,+\infty[,&\text{when }y=1.
	\end{cases}
	\]
	Clearly the map $L^<_\Theta$ is not lower semicontinuous at $(0,1)$. 
	In addition, the strict normal operator associated to $\Theta$ is given by
	\[
	\mathscr{N}^<_\Theta(x,y)=
	\begin{cases}
		[0,+\infty[,&\text{when }y\neq 1\\
		]-\infty,0],&\text{when }y=1.
	\end{cases}
	\]
	which is not closed.
\end{example}

\section{GNEPs and Variational Inequality Problems}\label{main-section}
Throughout this section, we will consider a jointly convex GNEP, given by a convex set $\mathfrak{X}\subset\R^n$. We will also consider the loss functions $\theta_\nu$ as bifunctions from $\R^{n_\nu} \times \R^{-n_\nu}$ to $\R$.

For each $\nu\in N$ and $x\in\R^n$, consider the set~\cite{Aussel-Dutta,SN10}
\[
D_{\nu}(x)=\conv\displaystyle\left( \mathscr{N}^<_{\theta_\nu}(x^{\nu},x^{-\nu})\cap S_\nu[0,1]\right),
\]
where $S_\nu[0,1]$ is the unit sphere in $\R^{n_\nu}$. This allows us to define the set-valued map $T:\R^n\tos\R^n$ as
\begin{equation}\label{T}
	T(x):=\prod_{\nu\in N}D_\nu(x).
\end{equation}
Clearly, for all $x\in\R^n$, $T(x)$ is compact and convex, as each $D_\nu(x)$ is compact and convex aswell.
However, the operator $T$ may not be a closed operator, as shown by the following example.

\begin{example}
	Consider a two player NEP where each player $\nu\in\{1,2\}$ has a loss function defined as in Example~\ref{E1}, namely
	\[
	\theta_\nu(x^{\nu},x^{-\nu})=\Theta(x^{\nu},x^{-\nu})=
	\begin{cases}
		x^\nu,&\text{when }x^{-\nu}\neq 1,\\
		-x^\nu,&\text{when }x^{-\nu}=1.
	\end{cases}
	\]
	Therefore 
	\[
	\mathscr{N}^<_{\theta_\nu}(x)=
	\begin{cases}
		[0,+\infty[,&\text{when }x^{-\nu}\neq 1,\\
		]-\infty,0],&\text{when }x^{-\nu}= 1,
	\end{cases}
	\]
	and
	\[
	D_\nu(x)=
	\begin{cases}
		\{1\},&\text{when }x^{-\nu}\neq 1,\\
		\{-1\},&\text{when }x^{-\nu}=1,
	\end{cases}
	\]
	which is not closed, so neither is $T$. 
\end{example}

In view of the previous example, we need additional conditions to guarantee the closedness of $T$. 

\begin{proposition}\label{closed-T}
	Let $T:\R^n\tos\R^n$ be defined as in~\eqref{T}. 
	\begin{enumerate}
		\item If each loss function $\theta_\nu$ is lower pseudocontinuous on its own player's variable and continuous with respect to its rivals' variables, then $T$ is closed.
		\item If each loss function $\theta_\nu$ is quasiconvex with respect to its own player's variable then $T$ is non-empty valued.
	\end{enumerate}
\end{proposition}
\begin{proof}
	\begin{enumerate}
		\item From Proposition~\ref{Prop1}, item 1, we have that each map $\mathscr{N}^<_{\theta_\nu}$ is closed with convex values.
		This implies that the map $D_\nu$ is closed with convex values. The result follows from the fact that the Cartesian product of closed maps is a closed map.
		\item In view of Remark~\ref{rem:finite}, we can use Proposition~\ref{Prop1}, item 3, to conclude that the set $\mathscr{N}^<_{\theta_\nu}(x^\nu,x^{-\nu})\setminus\{0\}$ is non-empty, for all $(x^\nu,x^{-\nu})\in \R^n$.
		Thus, the map $D_\nu$ is non-empty valued and so is $T$.
	\end{enumerate}
\end{proof}



Let $S(T,\mathfrak X)$ be the solution set of the Variational Inequality Problem  associated to the map $T$ (given as in \eqref{T}) and the set $\mathfrak X$, namely
\[
S(T,\mathfrak X)=\{x\in \mathfrak X\::\:\exists\,w_0\in T(x),\,\langle w_0,y-x\rangle\geq 0,\,\forall\,y\in\mathfrak X\}
\]
The following result establishes a link between the jointly convex GNEP and its associated variational inequality problem. 
\begin{proposition}\label{P1}
	Assume that $\mathfrak{X}$ is non-empty. If every loss function $\theta_\nu$
	is upper pseudocontinuous on its own player's variable, then every point in $S(T,\mathfrak{X})$ is a generalized Nash equilibrium.
\end{proposition}
\begin{proof}
	Take $x\in S(T,\mathfrak X)$, so there exists $w_0\in T(x)$ such that
	\begin{align}\label{eq:prop33}
		\inner{w_0}{y-x}\geq 0,\quad\forall\,y\in \mathfrak X. 
	\end{align}
	Assume that $x$ is not a solution of the jointly convex GNEP. Then there exist a player $\nu$, and $z^{\nu}\in \mathfrak{X}_\nu(x^{-\nu})$, such that 
	\[
	\theta_\nu(z^{\nu},x^{-\nu})<\theta_\nu(x^{\nu},x^{-\nu}),
	\]
	that is $z^{\nu}\in L^<_{\theta_\nu}(x^{\nu},x^{-\nu})$ and $y=(z^\nu,x^{-\nu})\in\mathfrak{X}$.  In view of~\eqref{eq:prop33},
	\[
	\inner{w_0^{\nu}}{z^{\nu}-x^{\nu}}=\inner{w_0}{y-x}\geq 0.
	\]
	This in turn implies that $w_0^\nu=0$, by Proposition~\ref{Prop1}, item 2, and the fact that $w_0^{\nu}\in \mathscr{N}^<_{\theta_\nu}(x^{\nu},x^{-\nu})$.
	
	On the other hand, since $w_0^{\nu}\in D_{\nu}(x)$, there exist $w^{\nu}_1,\ldots,w^{\nu}_p\in \mathscr{N}^<_{\theta_\nu}(x^{\nu},x^{-\nu})\cap S_{\nu}[0,1]$, $t_1,\ldots,t_p\geq 0$, $\displaystyle\sum_{i=1}^pt_i=1$, such that $0=w_0^\nu=\displaystyle\sum_{i=1}^pt_iw_i^\nu$. Take $i_0$ such that $t_{i_0}>0$, so we have
	\[
	0=\sum_{i\neq i_0}t_iw_i^\nu+t_{i_0}w^\nu_{i_0}
	\]
	and this implies 
	\[
	-w_{i_0}^\nu=\sum_{i\neq i_0} \frac{t_i}{t_{i_0}}w_i^\nu.
	\]
	As $\mathscr{N}^<_{\theta_\nu}(x)$ is a convex cone, $-w_{i_0}^\nu\in \mathscr{N}^<_{\theta_\nu}(x)$, hence
	\[
	w_{i_0}^\nu\in \mathscr{N}^<_{\theta_\nu}(x)\cap -\mathscr{N}^<_{\theta_\nu}(x).
	\]
	However, since $L^<_{\theta_\nu}(x)$ is open, Lemma~\ref{L0} implies that $w_{i_0}=0$, a contradiction. The proposition follows.
\end{proof}

Proposition~\ref{P1} is strongly related to~\cite[Theorem 3.1]{Aussel-Dutta}. 
However, the authors considered  continuity instead of upper pseudocontinuity. 

\begin{proposition}\label{P2}
	Assume that $\mathfrak{X}$ is convex, compact and non-empty and let $T$ be defined as in~\eqref{T}. If for all $\nu\in N$ the following hold
	\begin{enumerate}
		\item the function $\theta_\nu$ is lower pseudocontinuous on its own player's variable and continuous with respect to its rivals' variables,
		\item the function $\theta_\nu$ is quasiconvex with respect to its own player's variable;
	\end{enumerate}
	then $S(T,\mathfrak{X})$ is non-empty.
\end{proposition}
\begin{proof}
	Thanks to Proposition \ref{closed-T}, the map $T$ is closed with convex, compact and nonempty values. Thus, $T$ is upper semicontinuous. Finally, the result follows from 
	\cite[Theorem 9.9]{JP-Aubin-1998}.
\end{proof}

Finally, we are ready for our main result, which establishes the existence of solution for jointly convex GNEPs.

\begin{theorem} \label{P-theo}
	Assume that $\mathfrak{X}$ is convex, compact and non-empty. If for all $\nu\in N$ the following hold 
	\begin{enumerate}
		\item the function $\theta_\nu$ is pseudocontinuous on its own player's variable and continuous on its rivals' variables,
		\item the function $\theta_\nu$ is quasiconvex on its own player's variable,
	\end{enumerate}
	then there exists a generalized Nash equilibrium.
\end{theorem}
\begin{proof}
	It is  a consequence of Propositions \ref{P1} and \ref{P2}.
\end{proof}

\begin{remark} Theorem~\ref{P-theo} is not a consequence of \cite[Theorem 2]{Arrow-Debreu}. In fact, Theorem~\ref{P-theo} improves \cite[Theorem 4.2]{Aussel-Dutta}, \cite[Theorem 2.1]{FACCHINEI2007159}, \cite[Theorem 1]{Rosen} and \cite[Theorem 1]{Dansupta}\end{remark}

The following example shows that we cannot drop the continuity of each loss function with respect to its rivals' variables.
\begin{example}
Given the functions $\theta_1,\theta_2:\R^2\to\R$  defined as
\[
\theta_1(x^1,x^2):=\left\lbrace\begin{matrix}
	\left(x^1-\frac{\sqrt{2}}{2}\right)^2,&x^2\in\Q\\
	(2x^1-x^2)^2,&x^2\notin \Q
\end{matrix}\right.\mbox{ and }\theta_2(x^1,x^2):=\left\lbrace
\begin{matrix}
	\left(x^2-x^1\right)^2,&x^1\notin\Q\\
	\left(2x^2-x^1\right)^2,&x^1\in\Q
\end{matrix}\right.
\]
Clearly, each loss function is convex and continuous with respect to its own variable. 
For $X=[0,1]^2$, the GNEP reduces to the classic Nash equilibrium problem.  Furthermore, it is no difficult to see that there is not solution to this GNEP.
\end{example}
We must note that the previous example directly contradicts \cite[Corollary 4.3]{SN10}.



\begin{thebibliography}{10}

\bibitem{quasiconvex}
S.~Al-Homidan, N.~Hadjisavvas, and L.~Shaalan.
\newblock Transformation of quasiconvex functions to eliminate local minima.
\newblock {\em J. Optim. Theory and Appl.}, 177(1):93--105, 2018.

\bibitem{MR3236910}
S.~A.~R. Al-Mezel, F.~R.~M. Al-Solamy, and Q.~H. Ansari, editors.
\newblock {\em Fixed point theory, variational analysis, and optimization}.
\newblock CRC Press, Boca Raton, FL, 2014.

\bibitem{Arrow-Debreu}
K.~J. Arrow and G.~Debreu.
\newblock Existence of an equilibrium for a competitive economy.
\newblock {\em Econometrica}, 22(3):265--290, 1954.

\bibitem{JP-Aubin-1998}
J.~P. Aubin.
\newblock {\em Optima and equilibria : an introduction to nonlinear analysis /
  Jean-Pierre Aubin ; translated from the French by Stephen Wilson.}
\newblock Graduate texts in mathematics, 140. Springer, Berlin ;, 2nd ed.
  edition, 1998.

\bibitem{Aussel-Dutta}
D.~Aussel and J.~Dutta.
\newblock {Generalized Nash Equilibrium Problem, Variational Inequality and
  Quasiconvexity}.
\newblock {\em Oper. Res. Lett.}, 36(4):461--464, 2008.

\bibitem{AGM-16}
D.~Aussel, R.~Gupta, and A.~Mehra.
\newblock {Evolutionary Variational Inequality Formulation of the Generalized
  Nash Equilibrium Problem}.
\newblock {\em J. Optim. Theory Appl.}, 169(1):74--90, 2016.

\bibitem{Au-Su}
D.~Aussel and A.~Sultana.
\newblock {Quasi-variational inequality problems with non-compact valued
  constraint maps}.
\newblock {\em J. Math Anal. Appl.}, 456(2):1482--1494, 2017.

\bibitem{ASV-2016}
D.~Aussel, A.~Sultana, and V.~Vetrivel.
\newblock {On the existence of projected solutions of quasi-variational
  inequalities and generalized Nash equilibrium problems}.
\newblock {\em J. Optim. Theory Appl.}, 2016.

\bibitem{Banks-Duggan}
J.~Banks and J.~Duggan.
\newblock {Existence of Nash Equilibria on Convex Sets}.
\newblock Technical report, 2004.

\bibitem{Bueno-Cotrina}
O.~Bueno and J.~Cotrina.
\newblock On maximality of quasimonotone operators.
\newblock {\em Set-Valued and Variational Analysis}, 27(1):87--101, 2019.

\bibitem{CKK-2004}
J.~{Contreras}, M.~{Klusch}, and J.~B. {Krawczyk}.
\newblock Numerical solutions to nash-cournot equilibria in coupled constraint
  electricity markets.
\newblock {\em IEEE Transactions on Power Systems}, 19(1):195--206, 2004.

\bibitem{CK-2015}
J.~Contreras, J.~B. Krawczyk, and J.~Zuccollo.
\newblock Economics of collective monitoring: a study of environmentally
  constrained electricity generators.
\newblock {\em Computational Management Science}, 13(3):349--369, 2016.

\bibitem{Cotrina2020ExistenceOQ}
J.~Cotrina, A.~Hantoute, and A.~Svensson.
\newblock Existence of quasi-equilibria on unbounded constraint sets.
\newblock {\em Optimization}, pages 1--18, 2020.

\bibitem{JC-AS}
J.~Cotrina and A.~Svensson.
\newblock The finite intersection property for equilibrium problems.
\newblock {\em J Global Optim}, 2020.

\bibitem{JC-JZ}
J.~{Cotrina} and J.~{Z{\'u}{\~n}iga}.
\newblock {Time-dependent generalized Nash equilibrium problem}.
\newblock {\em J. Optim. Theory Appl.}, 179:1054--1064, 2018.

\bibitem{JC-JZ-PS}
J.~Cotrina and J.~Z{\'u}{\~n}iga.
\newblock {Quasi-equilibrium problems with non-self constraint map}.
\newblock {\em J. Glob Optim}, 75:177--197, 2019.

\bibitem{Cubiotti}
P.~Cubiotti.
\newblock Existence of nash equilibria for generalized games without upper
  semicontinuity.
\newblock {\em International Journal of Game Theory}, 26(2):267--273, 1997.

\bibitem{Dansupta}
P.~Dasgupta and E.~Maskin.
\newblock The existence of equilibrium in discontinuous economic games, part i
  (theory).
\newblock {\em Review of Economic Studies}, 53(1):1--26, 1986.
\newblock Reprinted in K. Binmore and P. Dasgupta (eds.), Economic
  Organizations as Games, Oxford: Basil Blackwell, 1986, pp. 48-82.

\bibitem{FACCHINEI2007159}
F.~Facchinei, A.~Fischer, and V.~Piccialli.
\newblock On generalized nash games and variational inequalities.
\newblock {\em Oper. Res. Let}, 35(2):159 -- 164, 2007.

\bibitem{VMPOS}
A.~G{\"o}pfert, H.~Riahi, C.~Tammer, and C.~Z\u{a}linescu.
\newblock {\em {Variational Methods in Partially Ordered Spaces}}.
\newblock {CMS Books in Mathematics}. Springer, New York, NY, 2003.

\bibitem{HARKER199181}
P.~T. Harker.
\newblock {Generalized Nash games and quasi-variational inequalities}.
\newblock {\em European Journal of Operational Research}, 54(1):81 -- 94, 1991.

\bibitem{PQK-NHQ-19}
P.~Q. Khanh and N.~H. Quan.
\newblock {Versions of the Weierstrass theorem for bifunctions and solution
  existence in optimization}.
\newblock {\em SIAM J. Optim}, 29(2):1502--1523, 2019.

\bibitem{K-2005}
J.~B. Krawczyk.
\newblock Coupled constraint nash equilibria in environmental games.
\newblock {\em Resource and Energy Economics}, 27(2):157 -- 181, 2005.

\bibitem{K-2016}
J.~B. Krawczyk and M.~Tidball.
\newblock Economic problems with constraints: How efficiency relates to
  equilibrium.
\newblock {\em International Game Theory Review}, 18(04):1650011, 2016.

\bibitem{MORGAN2007}
J.~Morgan and V.~Scalzo.
\newblock Pseudocontinuous functions and existence of nash equilibria.
\newblock {\em Journal of Mathematical Economics}, 43(2):174--183, 2007.

\bibitem{Nash}
J.~Nash.
\newblock Non-cooperative games.
\newblock {\em Annals of Mathematics}, 54(2):286--295, 1951.

\bibitem{SN10}
M.~Nasri and W.~Sosa.
\newblock {Equilibrium problems and generalized Nash games}.
\newblock {\em Optimization}, 60:1161--1170, 2011.

\bibitem{Rabia}
R.~Nessah and G.~Tian.
\newblock {Existence of Solution of Minimax Inequalities, Equilibria in Games
  and Fixed Points Without Convexity and Compactness Assumptions}.
\newblock {\em J. Optim. Theory Appl.}, 157:756--95, 2013.

\bibitem{Kazuo}
K.~Nishimura and J.~Friedman.
\newblock Existence of nash equilibrium in n person games without
  quasi-concavity.
\newblock {\em International Economic Review}, 22(3):637--648, 1981.

\bibitem{Reny}
P.~J. Reny.
\newblock {On the Existence of Pure and Mixed Strategy Nash Equilibria in
  Discontinuous Games}.
\newblock {\em Econometrica}, 67(5):1029--1056, September 1999.

\bibitem{Rosen}
J.~B. Rosen.
\newblock Existence and uniqueness of equilibrium points for concave n-person
  games.
\newblock {\em Econometrica}, 33(3):520--534, 1965.

\bibitem{Tian95}
G.~Tian and J.~Zhou.
\newblock {Transfer continuities, generalizations of the Weierstrass and
  maximum theorems: a full characterization}.
\newblock {\em J. Math. Econ.}, 24:281--303, 1995.

\bibitem{Tian92}
G.~Q. Tian.
\newblock {Generalization of the KKM Theorem and the Ky Fan Minimax Inequality,
  with Applications to Maximal elements, Price equilibrium, and Complementary}.
\newblock {\em J. Math. Anal. Appl.}, 170:457--471, 1992.

\bibitem{vardar2019}
B.~Vardar and G.~Zaccour.
\newblock Strategic bilateral exchange of a bad.
\newblock {\em Operations Research Letters}, 47(4):235 -- 240, 2019.

\end{thebibliography}
\end{document}